\documentclass[11pt]{amsart}

\usepackage{mg_macros}

\usepackage{graphicx}

\begin{document}


\title[Inference in $\alpha$-Brownian bridge]{Inference in $\alpha$-Brownian bridge \\ based on Karhunen-Lo\`eve expansions}

\author{Maik G\"orgens}
\address{Department of Mathematics, Uppsala University}
\curraddr{P.O. Box 480, 751 06 Uppsala, Sweden}
\email{maik@math.uu.se}

\date{\today}

\begin{abstract}
  We study a simple decision problem on the scaling parameter in the $\alpha$-Brownian bridge~$X^{(\alpha)}$ on the interval~$[0,1]$: given two values $\alpha_0, \alpha_1 \geq 0$ with $\alpha_0 + \alpha_1 \geq 1$ and some time $0 \leq T \leq 1$ we want to test $H_0: \alpha = \alpha_0$ vs. $H_1: \alpha = \alpha_1$ based on the observation of~$X^{(\alpha)}$ until time $T$. The likelihood ratio can be written as a functional of a quadratic form $\psi(X^{(\alpha)})$ of~$X^{(\alpha)}$. In order to calculate the distribution of~$\psi(X^{(\alpha)})$ under the null hypothesis, we generalize the Karhunen-Lo\`eve Theorem to positive finite measures on~$[0,1]$ and compute the Karhunen-Lo\`eve expansion of~$X^{(\alpha)}$ under such a measure. Based on this expansion, the distribution of~$\psi(X^{(\alpha)})$ follows by Smirnov's formula.
\end{abstract}


\maketitle


\section{Introduction}\label{S:Intro}

We consider the stochastic differential equation
\equ[E:aBB_SDE]{ dX^{(\alpha)}_t = dW_t - \frac{\alpha X^{(\alpha)}_t}{1-t} dt, \qquad X^{(\alpha)}_0 = 0, \quad 0 \leq t < 1, }
where $\alpha \geq 0$ and $W = (W_t)_{t \in [0,1]}$ is standard Brownian motion. We assume that $W$ is  defined on the probability space $(C([0,1]), \mathcal{C}, \Prob)$, where $C([0,1])$ is the space of continuous functions on the interval $[0,1]$ equipped with the supremum norm, $\mathcal{C}$ denotes the Borel sets, and $\Prob$ is the Wiener measure. Let $(\F_t)_{t \in [0,1]}$ be the natural filtration induced by $W$. The unique strong solution of~\eqref{E:aBB_SDE} is given by $X^{(\alpha)}=(X^{(\alpha)}_t)_{t \in [0,1)}$ with
\equ[E:aBB_Sol]{ X^{(\alpha)}_t = \int_0^t \left( \frac{1-t}{1-s} \right)^\alpha dW_s, \qquad 0 \leq t < 1. }
For $\alpha > 0$ we have $\lim_{t \rightarrow 1} X^{(\alpha)}_t = 0$ almost surely and thus $X^{(\alpha)}$ has an extension on $[0,1]$ with $X^{(\alpha)}_1 = 0$. The process $X^{(\alpha)}$ is called the $\alpha$-Brownian bridge with scaling parameter $\alpha$.

The $\alpha$-Brownian bridge is a mean reverting process, i.e., if $X^{(\alpha)}$ deviates from its mean $0$ at some time $0 < t < 1$, it is forced to return to $0$. The scaling parameter $\alpha$ determines how strong this force is. In order to examine this behavior further we compute the ``expected future'': for $0 \leq s \leq t \leq 1$ we have
\begin{align*} 
  \E [ X^{(\alpha)}_t \mid \F_s ] &= \E \left[ \int_0^t \left( \frac{1-t}{1-x} \right)^\alpha dW_x \Big| \F_s \right] \\ 
    &= \left( \frac{1-t}{1-s} \right)^\alpha \int_0^s \left( \frac{1-s}{1-x} \right)^\alpha dW_x \\
    &= \left( \frac{1-t}{1-s} \right)^\alpha X^{(\alpha)}_s.
\end{align*}
Again, we see that the smaller the scaling parameter $\alpha$ is the more the process will deviate from its mean $0$. In the case $\alpha = 0$ we obtain standard Brownian motion, i.e., $X^{(1)} = W$ and $\E [ X^{(0)}_t \mid \F_s ] = X^{(0)}_s$, and in the case $\alpha = 1$ we obtain the usual Brownian bridge with $X^{(1)}_0 = X^{(1)}_1 = 0$ and
\[ \E [ X^{(1)}_t \mid \F_s ] = \frac{1-t}{1-s} X^{(1)}_s. \]
\begin{figure}
  \includegraphics[width=1\textwidth]{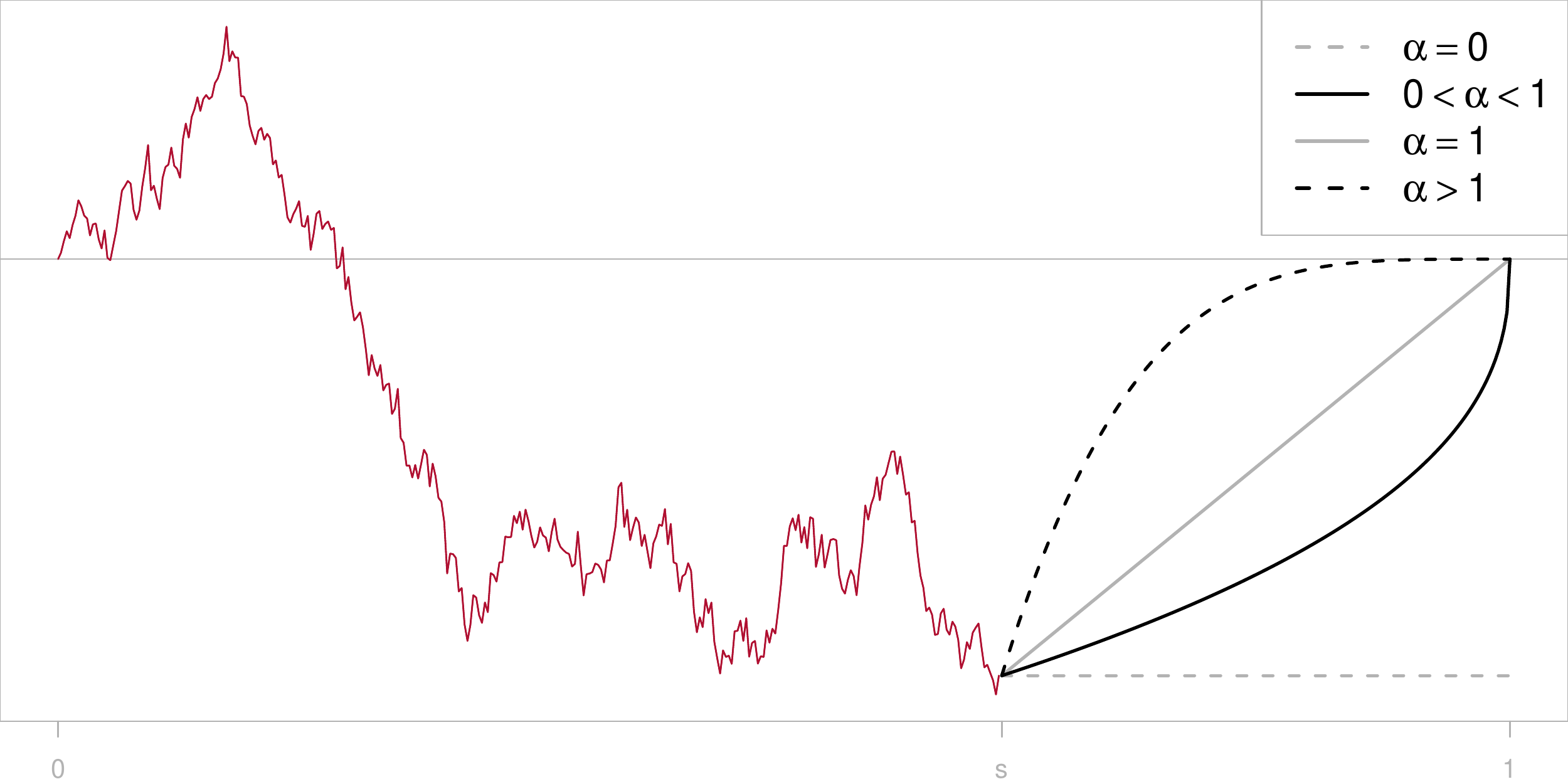}
  \caption{The influence of $\alpha$ to the ``expected future'' $\E [ X^{(\alpha)}_t \mid \F_s ]$ for different values of $\alpha$.}
\end{figure}

In this paper we assume that the scaling parameter $\alpha$ is unknown and, given two different values $\alpha_0, \alpha_1 \geq 0$ with $\alpha_0 + \alpha_1 \geq 1$ and some time $0 \leq T \leq 1$, we want to test
\equ[E:Dec_Prob]{ H_0: \alpha = \alpha_0 \qquad \text{vs.} \qquad H_1: \alpha = \alpha_1, }
based on an observed trajectory of $X^{(\alpha)}$ until time $T$, i.e., the decision should be based on the information in $\F_T$. When deciding problem~\eqref{E:Dec_Prob} we can make two types of error. Rejecting the hypothesis $H_0$ though $\alpha = \alpha_0$ is true we make an error of the first kind, whereas keeping the hypothesis $H_0$ though $\alpha = \alpha_1$ is true we make an error of the second kind. Our aim is to find that decision which minimizes the probability of making an error of the second kind, given that the probability of making an error of the first kind is not larger than $q$ for some $0 \leq q \leq 1$. The Neyman--Pearson Lemma yields the most powerful test (see~\cite{Lie08} or other introductory texts on statistical decision theory): let $\Prob^{(\alpha)}$ be the induced measure of $X^{(\alpha)}$ on the filtered measurable space $(C([0,1]), \mathcal{C}, (\F_t)_{t \in [0,1]})$ (note in particular that $\Prob^{(0)} = \Prob$) and let $\Prob^{(\alpha)}_t$ denote the restriction of the probability measure 
$\Prob^{(\alpha)}$ to the $\sigma$-algebra $\F_t$, $t \in [0,1]$. Assume $T < 1$ (the case $T=1$ is treated separately in Section~\ref{SS:T1}). Then we have to decide according to the following rule:
\[ \text{reject $H_0$ if $\phi_{\alpha_0, \alpha_1}(T) > c_{\alpha_0, \alpha_1, T}(q)$,} \]
where $\phi_{\alpha_0, \alpha_1}(T) := d\Prob^{(\alpha_1)}_T / d\Prob^{(\alpha_0)}_T$ is the likelihood ratio at time $T$ and $c_{\alpha_0, \alpha_1, T}(q)$ is chosen such that
\[ \Prob^{(\alpha_0)}(\phi_{\alpha_0, \alpha_1}(T) > c_{\alpha_0, \alpha_1, T}(q)) = q. \]
Knowing the distribution of $\phi_{\alpha_0, \alpha_1}(T)$ under $\Prob^{(\alpha_0)}$ is thus crucial in finding the optimal decision in the statistical decision problem~\eqref{E:Dec_Prob}.

In Section~\ref{SS:LRP} we will show
\begin{prop}\label{P:LRP}
  The likelihood ratio $\phi_{\alpha_0, \alpha_1}(T)$ is given by
  \equ[E:Phi]{ \phi_{\alpha_0, \alpha_1}(T) = \exp \left( (\alpha_0 - \alpha_1) ( \psi_{\alpha_0, \alpha_1}(T) +  \ln(1-T) )/2 \right), }
  where
  \equ[E:Psi]{ \psi_{\alpha_0, \alpha_1}(T) = \frac{(X^{(\alpha)}_T)^2}{1-T} + (\alpha_0 + \alpha_1 - 1) \int_0^T \frac{(X^{(\alpha)}_s)^2}{(1-s)^2} ds. }
\end{prop}
According to~\eqref{E:Phi} it is enough to determine the distribution of $\psi_{\alpha_0, \alpha_1}(T)$ under $\Prob^{(\alpha_0)}$. Then the distribution of $\phi_{\alpha_0, \alpha_1}(T)$ follows by simple transformations.

We introduce the measure $\mu_{\alpha_0, \alpha_1, T}$ by
\equ[E:Mu]{ \mu_{\alpha_0, \alpha_1, T}(ds) := \frac{\delta_T(ds)}{1-T} + \frac{(\alpha_0+\alpha_1-1) \Ind(s \leq T) ds}{(1-s)^2}, }
where $\delta_T$ denotes the point measure at $T$ and $\Ind$ the indicator function. By the assumption $\alpha_0 + \alpha_1 \geq 1$ it follows that $\mu_{\alpha_0, \alpha_1, T}$ is a positive measure. Let $L_2(\mu_{\alpha_0, \alpha_1, T})$ denote the space of functions on $[0,1]$ that are square integrable with respect to the measure $\mu_{\alpha_0, \alpha_1, T}$. From~\eqref{E:Psi} we see that $\psi_{\alpha_0, \alpha_1}(T)$ is the squared $L_2$-norm of $X^{(\alpha)}$ under the measure $\mu_{\alpha_0, \alpha_1, T}$, i.e.,
\equ[E:Psi_Mu]{ \psi_{\alpha_0, \alpha_1}(T) = \| X^{(\alpha)} \|^2_{L_2(\mu_{\alpha_0, \alpha_1, T})}. }

The covariance function $R^{(\alpha)}(s,t) := \E[ X^{(\alpha)}_s X^{(\alpha)}_t ]$ of $X^{(\alpha)}$ is given by
\equ[E:Covar]{ R^{(\alpha)}(s,t) = \frac{(1-s)^\alpha (1-t)^\alpha}{1-2\alpha}(1-(1-(s \wedge t))^{1-2\alpha}) }
for $\alpha \neq 1/2$ and
\[ R^{(\alpha)}(s,t) = - \sqrt{(1-s)(1-t)} \ln(1- (s \wedge t)) \]
for $\alpha = 1/2$, where $s \wedge t$ denotes the minimum of $s$ and $t$. With $R^{(\alpha_0)}$ we associate the integral operator $A_{R^{(\alpha_0)}}$ defined by
\equ[E:Int_Op]{ (A_{R^{(\alpha_0)}}e)(t) = \int_0^1 R^{(\alpha_0)}(t,s) e(s) \mu_{\alpha_0, \alpha_1, T}(ds). }
For $T < 1$, we have
\equ[E:Finite_Int]{ \int_0^1 \int_0^1 |R^{(\alpha_0)}(t,s)|^2 \mu_{\alpha_0, \alpha_1, T}(dt) \mu_{\alpha_0, \alpha_1, T}(ds) < \infty. }
Hence, by the Cauchy-Schwartz inequality,
\begin{align*}
  &\| A_{R^{(\alpha_0)}}e \|^2_{L_2(\mu_{\alpha_0, \alpha_1, T})} \\
    &\qquad= \int_0^1 \left| \int_0^1 R^{(\alpha_0)}(t,s) e(s) \mu_{\alpha_0, \alpha_1, T}(ds) \right|^2 \mu_{\alpha_0, \alpha_1, T}(dt) \\
    &\qquad\leq \|e\|^2_{L_2(\mu_{\alpha_0, \alpha_1, T})} \int_0^1 \int_0^1 | R^{(\alpha_0)}(t,s)|^2 \mu_{\alpha_0, \alpha_1, T}(ds) \mu_{\alpha_0, \alpha_1, T}(dt) < \infty \\
\end{align*}
for $e \in L_2(\mu_{\alpha_0, \alpha_1, T})$ which implies that $A_{R^{(\alpha_0)}}$ is a linear and bounded operator from $L_2(\mu_{\alpha_0, \alpha_1, T})$ to $L_2(\mu_{\alpha_0, \alpha_1, T})$ with
\[ \| A_{R^{(\alpha_0)}} \|^2 \leq \int_0^1 \int_0^1 | R^{(\alpha_0)}(t,s)|^2 \mu_{\alpha_0, \alpha_1, T}(ds) \mu_{\alpha_0, \alpha_1, T}(dt) < \infty. \]
Moreover, from~\eqref{E:Finite_Int} it follows that $A_{R^{(\alpha_0)}}$ is compact, the symmetry of $R^{(\alpha_0)}$ implies the self-adjointness of $A_{R^{(\alpha_0)}}$, and since $R^{(\alpha_0)}$ is non-negative definite it follows that $A_{R^{(\alpha_0)}}$ is non-negative definite. Hence, its eigenvalues $(\lambda_k)_{k=1}^\infty$ are real and non-negative and an application of a generalized version of the Karhunen-Lo\`eve Theorem (see Section~\ref{SS:GKLT}) yields the following series expansion of $X^{(\alpha_0)}$:
\equ[E:KLE]{ X^{(\alpha_0)}_t = \sum_{k=1}^\infty Z_k e_k(t), }
where $(e_k)_{k=1}^\infty$ is the sequence of corresponding orthonormalized eigenfunctions of the eigenvalues $(\lambda_k)_{k=1}^\infty$ and $(Z_k)_{k=1}^\infty$ is a sequence of independent normal random variables with $\E Z_k^2 = \lambda_k$. The convergence in~\eqref{E:KLE} is almost surely uniform in $t$ for all $t \in [0,T]$.

From the bi-orthogonality in~\eqref{E:KLE}, i.e., independent random variables $Z_k$ and orthogonal eigenfunctions $e_k$, we obtain the following distributional equivalence for $\psi_{\alpha_0, \alpha_1}(T)$ under $\Prob^{(\alpha_0)}$: by~\eqref{E:Psi_Mu} and~\eqref{E:KLE} we have
\equ[E:Psi_Dist_Equ]{ \psi_{\alpha_0, \alpha_1}(T) = \| X^{(\alpha_0)} \|^2_{L_2(\mu_{\alpha_0, \alpha_1, T})} = \sum_{k=1}^\infty Z_k^2 =_d \sum_{k=1}^\infty \lambda_k \mathcal{N}_k^2, }
where $=_d$ means equality in distribution and $(\mathcal{N}_k)_{k=1}^\infty$ is an i.i.d. sequence of standard normal random variables. 

Random variables of the form
\[ Q_r = \sum_{k=1}^r \nu_k \mathcal{N}_k^2 \]
with $\nu_k > \nu_l \geq 0$ for $k < l$ and $\mathcal{N}_k$ as above were studied by Smirnov in~\cite{Smi37}. In~\cite{Mar75} it was proven that the formula found by Smirnov extends to $r=\infty$ whenever $\sum_{k=1}^\infty \nu_k < \infty$. Namely, it was shown that
\equ[E:Dist_Q]{ \Prob(Q_\infty \leq x) = 1 - \frac{1}{\pi} \sum_{k=1}^\infty (-1)^{k+1} \int_{1/\nu_{2k-1}}^{1/\nu_{2k}} \frac{e^{-xu/2}}{u \sqrt{|F(u)|}} du, }
where $F$ is the real valued function
\[ F(u) = \prod_{l=1}^\infty ( 1 - \nu_l u ). \]

In Theorem~\ref{T:KLE} we calculate the eigenvalues $(\lambda_k)_{k=1}^\infty$ of the operator $A_{R^{(\alpha_0)}}$. In the case $\alpha_0, \alpha_1 \geq 1/2$ with
\[ \alpha_1 \neq 1 - \alpha_0 - \frac{(1-2\alpha_0)^2 \ln(1-T)}{2 (1-2\alpha_0) \ln(1-T) + 4} \]
these are given by the positive zeros of the function
\[ F_{\alpha_0, \alpha_1, T}(\lambda) := \tan(\beta(\lambda) \ln(1-T)) + \lambda \beta(\lambda) / (1 + \lambda/2 - \lambda \alpha_0), \]
where $\beta(\lambda) = \sqrt{(\alpha_0 + \alpha_1 - 1)/\lambda - \alpha_0(\alpha_0 - 1) - 1/4}$.
In the general case $\alpha_0, \alpha_1 \geq 0$ (but with $\alpha_0 + \alpha_1 \geq 1$) further eigenvalues in addition to the zeros of the function $F_{\alpha_0, \alpha_1, T}$ can appear (see Theorem~\ref{T:KLE}). We show that $\sum_{k=1}^\infty \lambda_k < \infty$ in Proposition~\ref{P:SUMMABLE}. Then, according to~\eqref{E:Psi_Dist_Equ}, the distribution function of $\psi_{\alpha_0, \alpha_1}(T)$ under $\Prob^{(\alpha_0)}$ is given by~\eqref{E:Dist_Q} with $\nu_k = \lambda_k$. Finally, from~\eqref{E:Phi} we obtain the following

\begin{theorem}
  If $\alpha_0 < \alpha_1$, then $\phi_{\alpha_0, \alpha_1}(T) \leq (1-T)^{(\alpha_0 - \alpha_1)/2}$ and the distribution function of $\phi_{\alpha_0, \alpha_1}(T)$ under $\Prob^{(\alpha_0)}$ is given by $\Prob^{(\alpha_0)}(\phi_{\alpha_0, \alpha_1}(T) \leq x) = D_{\alpha_0, \alpha_1, T}(x)$, where
  \[ D_{\alpha_0, \alpha_1, T}(x) = \frac{1}{\pi} \sum_{k=1}^\infty (-1)^{k+1} \int_{1/\lambda_{2k-1}}^{1/\lambda_{2k}} \frac{(1-T)^{u/2} x^{u/(\alpha_1 - \alpha_0)}}{u \sqrt{|F(u)|}} du \]
  with $F(u) = \prod_{l=1}^\infty ( 1 - \lambda_l u )$. In the case $\alpha_0 > \alpha_1$ we have $\phi_{\alpha_0, \alpha_1}(T) \geq (1-T)^{(\alpha_0 - \alpha_1)/2}$ and $\Prob^{(\alpha_0)}(\phi_{\alpha_0, \alpha_1}(T) \leq x) = 1 - D_{\alpha_0, \alpha_1, T}(x)$.
\end{theorem}

\begin{remark}
  We may as well define the $\alpha$-Brownian bridge on an interval $[0,S]$. Let $X^{(\alpha, S)}=(X^{(\alpha, S)}_t)_{t \in [0,S)}$ be the strong solution of the stochastic differential equation
  \[ dX^{(\alpha, S)}_t = dW_t - \frac{\alpha X^{(\alpha, S)}_t}{S-t} dt, \qquad X^{(\alpha, S)}_0 = 0, \quad 0 \leq t < S. \]
  Then, for $\alpha > 0$, we have $\lim_{t \rightarrow S} X^{(\alpha, S)}_t = 0$. The $\alpha$-Brownian bridge is self-similar. Namely,
  \[ \left( X^{(\alpha, S)}_t \right)_{t \in [0, S]} =_d \left( \sqrt{S} X^{(\alpha, 1)}_{t/S} \right)_{t \in [0, S]}. \]
  From this self-similarity the results in this paper easily extend to $\alpha$-Brownian bridges on an interval $[0,S]$. However, we do not pursue this case further.
\end{remark}

To the best of our knowledge, the $\alpha$-Brownian bridge was first studied in~\cite{Bre90}, where it was used to model the arbitrage profit associated with a given futures contract in the absence of transaction costs. In~\cite{Bar10} sample path properties of $X^{(\alpha)}$ and the maximum-likelihood estimator of $\alpha$ where studied. In~\cite{Bar11a} Laplace transforms of $X^{(\alpha)}$ are calculated. In particular, the Laplace transform of $\psi_{\alpha_0, \alpha_1}(T)$ follows from Theorem~21 in~\cite{Bar11a}. In~\cite{Bar11b}, the Karhunen-Lo\`eve expansion of $X^{(\alpha)}$ under the Lebesgue measure was computed. The decision problem~\ref{E:Dec_Prob} was studied before in~\cite{Zha13} under the assumption that $\alpha_0, \alpha_1 > 1/2$ and that the time of decision $T$ is close to $1$. An approximation of the distribution of $\phi_{\alpha_0, \alpha_1}(T)$ under $\Prob^{(\alpha_0)}$ was derived by means of large deviations. We improve those results by allowing a more general setting for the parameters $\alpha_0, \alpha_1$, and $T$, and by providing exact formulas for the distribution of the likelihood ratio under $H_0$.

The rest of the paper is organized as follows. In Section~\ref{S:Prel} we calculate the likelihood ratio $\phi_{\alpha_0, \alpha_1}(T)$ and study the cases $T=1$ and $\alpha_0 + \alpha_1 = 1$. We exclude these cases in the later sections. In Section~\ref{S:KLE} we generalize the Karhunen-Lo\`eve Theorem and calculate the Karhunen-Lo\`eve expansion of $X^{(\alpha_0)}$ under the measure $\mu_{\alpha_0,\alpha_1,T}$. In Section~\ref{S:OU} we briefly comment on how the approach of this paper extends to other processes, such as the Ornstein-Uhlenbeck process. Finally, in Section~\ref{S:Details} we give some remaining proofs we did not give in earlier sections for the sake of readability.


\section{Preliminary results and special cases}\label{S:Prel}

\subsection{The likelihood ratio process}\label{SS:LRP}

We prove Proposition~\ref{P:LRP}, i.e., we show
\[ \phi_{\alpha_0, \alpha_1}(T) = \exp \left( (\alpha_0 - \alpha_1) ( \psi_{\alpha_0, \alpha_1}(T) +  \ln(1-T) )/2 \right), \]
where
\[ \psi_{\alpha_0, \alpha_1}(T) = \frac{(X^{(\alpha)}_T)^2}{1-T} + (\alpha_0 + \alpha_1 - 1) \int_0^T \frac{(X^{(\alpha)}_s)^2}{(1-s)^2} ds. \]

\begin{proof}[Proof of Proposition~\ref{P:LRP}]
  Under $\Prob^{(0)}_T$ we have 
  \[ dX^{(\alpha)}_t = dW_t, \qquad X^{(\alpha)}_0 = 0, \quad t \leq T, \]
  i.e., $X^{(\alpha)}$ is Brownian motion. Under $\Prob^{(\alpha)}_T$ we get
  \[ dX^{(\alpha)}_t = dW_t - \frac{\alpha X^{(\alpha)}_t}{1-t} dt, \qquad X^{(\alpha)}_0 = 0, \quad t \leq T, \]
  and thus
  \[ dW_t = dX^{(\alpha)}_t + \frac{\alpha X^{(\alpha)}_t}{1-t} dt. \]
  
  For $0 \leq t \leq T$, define $M^{(\alpha)}_t$ by
  \begin{align} 
    M^{(\alpha)}_t &= \exp \left( - \int_0^t \left( - \frac{\alpha X^{(\alpha)}_s}{1-s} \right) dW_s - \frac{1}{2} \int_0^t \left( - \frac{\alpha X^{(\alpha)}_s}{1-s} \right)^2 ds \right) \label{E:LRP_1} \\
        &= \exp \left( \alpha \int_0^t \frac{X^{(\alpha)}_s}{1-s} dX^{(\alpha)}_s + \frac{\alpha^2}{2} \int_0^t \frac{(X^{(\alpha)}_s)^2}{(1-s)^2} ds \right). \notag
  \end{align}
  The process $(M^{(\alpha)}_t)_{t \in [0, T]}$ is a martingale with respect to $(\F_t)_{t \in [0, T]}$ and $\Prob^{(\alpha)}_T$ and thus, by Girsanov's Theorem, $X^{(\alpha)}$ is a Brownian motion on $\F_T$ with respect to the measure $\Q$ defined by $d \Q = M^{(\alpha)}_T d \Prob^{(\alpha)}_T$. Hence, $\Q = \Prob^{(0)}_T$ on $\F_T$ and thus
  \equ[E:LRP_2]{ \frac{d\Prob^{(\alpha)}_T}{d\Prob^{(0)}_T} = (M^{\alpha}_T)^{-1}. }
  It follows
  \equ[E:LRP_3]{\phi_{\alpha_0, \alpha_1}(T) = \frac{d\Prob^{(\alpha_1)}_T}{d\Prob^{(\alpha_0)}_T} =  M^{\alpha_0}_T / M^{\alpha_1}_T. }
  
  In order to calculate $M^{(\alpha)}_T$ set $Y_t = X^{(\alpha)}_t / (1-t)$, $0 \leq t \leq T$. Then, by It\^o's formula,
  \[ dY_t = \frac{dX^{(\alpha)}_t}{1-t} + \frac{X^{(\alpha)}_t}{(1-t)^2} dt, \quad Y_0 = 0, \quad 0 \leq t \leq T, \]
  and it follows by partial integration that
  \begin{align*}
    \int_0^T \frac{X^{(\alpha)}_s}{1-s} dX^{(\alpha)}_s &= \int_0^T Y_s dX^{(\alpha)}_s \\
      &= Y_T X^{(\alpha)}_T - Y_0 X^{(\alpha)}_0 - \int_0^T X^{(\alpha)}_s dY_s - \int_0^T dX^{(\alpha)}_s \cdot dY_s \\
      &= \frac{(X^{(\alpha)}_T)^2}{1-T} - \int_0^T \frac{X^{(\alpha)}_s}{1-s} dX^{(\alpha)}_s - \int_0^T \frac{(X^{(\alpha)}_s)^2}{(1-s)^2} ds - \int_0^T \frac{ds}{1-s},
  \end{align*}
  and thus that
  \equ[E:LRP_4]{ \int_0^T \frac{X^{(\alpha)}_s}{1-s} dX^{(\alpha)}_s = \frac{1}{2} \left( \frac{(X^{(\alpha)}_T)^2}{1-T} - \int_0^T \frac{(X^{(\alpha)}_s)^2}{(1-s)^2} ds + \ln(1-T) \right). }
  Plugging~\eqref{E:LRP_4} into~\eqref{E:LRP_1}, we obtain
  \equ[E:LRP_5]{ M^{(\alpha)}_T = \exp \left( \frac{\alpha (X^{(\alpha)}_T)^2}{2(1-T)} - \frac{\alpha}{2}(1-\alpha) \int_0^T \frac{(X^{(\alpha)}_s)^2}{(1-s)^2} ds + \frac{\alpha}{2} \ln(1-T) \right). }
  Finally, plugging~\eqref{E:LRP_5} into~\eqref{E:LRP_3} yields the desired result.
\end{proof}


\subsection{The case $T=1$}\label{SS:T1}

From~\eqref{E:LRP_2} and~\eqref{E:LRP_5} it follows that the maximum-likelihood estimator of $\alpha$ based on $\F_T$ is given by
\[ \hat \alpha_T = \left(- \frac{(X^{(\alpha)}_T)^2}{1-T} + \int_0^T \frac{(X^{(\alpha)}_s)^2}{(1-s)^2} ds - \ln(1-T)\right)/\left(2 \int_0^T \frac{(X^{(\alpha)}_s)^2}{(1-s)^2} ds\right). \]
It was shown in~\cite{Bar11a} that $\hat \alpha_T$ is a strongly consistent estimator for $\alpha$, i.e., we have $\lim_{T \rightarrow 1} \hat \alpha_T = \alpha$, $\Prob^{(\alpha)}$-almost surely. Hence, at time $T=1$ we can test~\eqref{E:Dec_Prob} without any risk of making an error of the first or the second kind. Therefore, in the remaining part of the paper we assume that $T < 1$.


\subsection{The case $\alpha_0 + \alpha_1 = 1$}

We will now study the case $\alpha_0 + \alpha_1 = 1$. From Proposition~\ref{P:LRP} we know
\[ \phi_{\alpha_0, \alpha_1}(T) = \exp \left( \frac{(\alpha_0 - \alpha_1)}{2} \left( \frac{(X^{(\alpha)}_T)^2}{1-T} + \ln(1-T) \right) \right). \]
If $\alpha_0 < \alpha_1$ then $\phi_{\alpha_0, \alpha_1}(T) \leq (1-T)^{(\alpha_0 - \alpha_1)/2}$ and since $X^{(\alpha_0)}_T$ is normally distributed with mean $0$ and variance $R^{(\alpha_0)}(T,T)$ it follows $\Prob^{(\alpha_0)}(\phi_{\alpha_0, \alpha_1}(T) \leq x) = D_{\alpha_0, \alpha_1, T}(x)$ , where
\[ D_{\alpha_0, \alpha_1, T}(x) = 2 - 2 \Phi\left(\sqrt{(1-T)\left(\frac{2\ln(x)}{\alpha_0-\alpha_1} - \ln(1-T)\right)/R^{(\alpha_0)}(T,T)}\right), \]
and
\[ \Phi(x) := \frac{1}{\sqrt{2 \pi}} \int_{-\infty}^x \exp(-y^2/2) dy \]
is the distribution function of the standard normal distribution. In the case $\alpha_0 > \alpha_1$ we obtain $\phi_{\alpha_0, \alpha_1}(T) \geq (1-T)^{(\alpha_0 - \alpha_1)/2}$ and then $\Prob^{(\alpha_0)}(\phi_{\alpha_0, \alpha_1}(T) \leq x) = 1 - D_{\alpha_0, \alpha_1, T}(x)$. 

In the particular case $\alpha_0 = 0$ and $\alpha_1 = 1$ we would like to distinguish Brownian bridge from Brownian motion. We have
\begin{align*}
  \phi_{0, 1}(T) &= \exp \left( - \frac{1}{2} \left( \frac{(X^{(\alpha)}_T)^2}{1-T} + \ln(1-T) \right) \right) \\
    &= \exp \left( - \frac{(X^{(\alpha)}_T)^2}{2(1-T)} \right) / \sqrt{1-T}.
\end{align*}
Moreover, $R^{(0)}(T,T) = T$ and thus
\[ \Prob^{(0)}(\phi_{0, 1}(T) \leq x) = 2 - 2 \Phi(\sqrt{(1-T)(-2\ln(x) - \ln(1-T))/T}) \]
for $x \leq 1 / \sqrt{1-T}$.

In the remaining part of the paper we assume that $\alpha_0 + \alpha_1 > 1$.


\section{A Karhunen-Lo\`eve expansion of $X^{(\alpha)}$}\label{S:KLE}

Let $Y = (Y_t)_{t \in [a,b]}$ be a centered continuous Gaussian process indexed by a compact interval $[a,b]$ with covariance function $R(s,t) = \E Y_s Y_t$. Let $(\lambda_k)_{k=1}^\infty$ and $(e_k)_{k=1}^\infty$ be the eigenvalues and corresponding orthonormalized eigenfunctions of the operator $A_R: L_2([a,b]) \rightarrow L_2([a,b])$ defined by
\equ[E:GKLT_0]{ (A_R e)(t) = \int_a^b R(s, t) e(s) ds, \quad e \in L_2([a,b]). }

The Karhunen-Lo\`eve Theorem (see Theorem~34.5.B in~\cite{Loe63}) implies that $Y$ has a series expansion of the form
\equ[E:GKLT_1b]{ Y_t = \sum_{k=1}^\infty Z_k e_k(t) \qquad \text{with} \qquad Z_k = \int_a^b Y_s e_k(s) ds, }
where $(Z_k)_{k=1}^\infty$ is a sequence of independent centered normal distributed random variables with $E Z_k^2 = \lambda_k$ and the convergence in~\eqref{E:GKLT_1b} is almost surely uniform in $t \in [a,b]$.

We extend this result by replacing the Lebesgue measure on the interval~$[a,b]$ by any positive finite Borel measure $\nu$ on~$[a,b]$. Then we calculate the Karhunen-Lo\`eve expansion of $X^{(\alpha)}$ under the measure $\mu_{\alpha_0, \alpha_1, T}$ defined in~\eqref{E:Mu}.

\subsection{A generalized Karhunen-Lo\`eve Theorem}\label{SS:GKLT}

Consider a continuous centered Gaussian process $Y = (Y_t)_{t \in [a,b]}$ and let $R(s,t) = \E Y_s Y_t$ be the covariance function of $Y$. Let $\nu$ be a positive finite Borel measure on~$[a,b]$ with support $C \subset [a,b]$. Then we have $R \in L_2([a,b]^2, \nu \otimes \nu)$ and the linear operator $A_R: L_2([a,b], \nu) \rightarrow L_2([a,b], \nu)$ defined by
\equ[E:GKLT_1]{ (A_R e)(t) = \int_a^b R(s, t) e(s) \nu(ds) }
is bounded, compact, self-adjoint, and non-negative definite. Hence, the eigenvalues $(\lambda_k)_{k=1}^\infty$ of $A_R$ are real and non-negative and we get the spectral decomposition
\equ[E:GKLT_2]{ A_R e = \sum_{k=1}^\infty \lambda_k \la e_k, e \ra e_k, }
where $e_k$ is the corresponding orthonormalized eigenfunction of the eigenvalue $\lambda_k$ and $\la \cdot, \cdot \ra$ denotes the scalar product of $L_2([a,b], \nu)$.

\begin{theorem}[Generalized Karhunen-Lo\`eve theorem]\label{T:GKLE}
  We have
  \equ[E:GKLT_3]{ Y_t = \sum_{k=1}^\infty Z_k e_k(t) \qquad \text{with} \qquad Z_k = \int_a^b Y_s e_k(s) \nu(ds), }
  where the convergence is almost surely uniform in $t$ for all $t \in C$. Moreover, the $(Z_k)_{k=1}^\infty$ form a sequence of independent centered normal distributed random variables with $E Z_k^2 = \lambda_k$.
\end{theorem}

\begin{remark}
  We may even replace the interval~$[a,b]$ by a topological Hausdorff space $E$ together with a finite Borel measure on $E$ with compact support and consider Gaussian processes indexed by $E$. However, we do not need this generality and thus we do not pursue this case further.
\end{remark}

The following proof is a modified version of the proof of the Karhunen-Lo\`eve Theorem in~\cite{Loe63}.

\begin{proof}
  From~\eqref{E:GKLT_1} and~\eqref{E:GKLT_2} it follows for all $e \in L_2([a,b], \nu)$ that
  \begin{align*}
    \int_a^b R(s, t) e(s) \nu(ds) &= \sum_{k=1}^\infty \lambda_k \la e_k, e \ra e_k(t) \\
      &= \sum_{k=1}^\infty \lambda_k \int_a^b e_k(s) e(s) \nu(ds) e_k(t) \\
      &= \int_a^b \left( \sum_{k=1}^\infty \lambda_k e_k(s) e_k(t) \right) e(s) \nu(ds),
  \end{align*}
  and thus
  \equ[E:GKLT_4]{ R(s, t) = \sum_{k=1}^\infty \lambda_k e_k(s) e_k(t) }
  for $\nu$-almost all $s,t \in C$. Moreover, by Mercer's theorem (see Theorem~3.a.1 in~\cite{Kon86}), the convergence in~\eqref{E:GKLT_4} is uniform in $s,t \in C$.
  
  We introduce
  \[ Y_t^{(n)} = \sum_{k=1}^n Z_k e_k(t) \]
  with $Z_k$ as in~\eqref{E:GKLT_3}. Then
  \begin{align} 
    E [ Y_t - Y_t^{(n)} ]^2 &= \E [ Y_t ] ^2 - 2 \E [ Y_t Y_t^{(n)} ] + \E [ Y_t^{(n)} ]^2 \label{E:GKLT_5} \\
      &= R(t, t) - 2 \sum_{k=1}^n e_k(t) \E [ Y_t Z_k ] + \sum_{k,l = 1}^n e_k(t) e_l(t) \E [ Z_k Z_l ]. \notag
  \end{align}
  Since $e_k$ is the corresponding eigenfunction of the eigenvalue $\lambda_k$ of $A_R$, we have
  \begin{align*}
    \E [ Y_t Z_k ] &= \int_a^b \E [Y_t Y_s] e_k(s) \nu(ds) = \int_a^b R(s,t) e_k(s) \nu(ds) \\
      &= (A_R e_k)(t) = \lambda_k e_k(t)
  \end{align*}
  and thus, with $\delta_{i,j}$ denoting the Kronecker symbol,
  \[ \E [ Z_k Z_l ] = \int_a^b \E [ Z_k Y_t ] e_l(t) \nu(dt) = \int_a^b \lambda_k e_k(t) e_l(t) \nu(dt) = \lambda_k \delta_{k,l}. \]
  By~\eqref{E:GKLT_5}, it follows
  \begin{align*}
    E [ Y_t - Y_t^{(n)} ]^2 &= R(t, t) - 2 \sum_{k=1}^n e_k(t) \lambda_k e_k(t) + \sum_{k = 1}^n e_k(t) e_k(t) \lambda_k \\
      &= R(t, t) - \sum_{k=1}^n \lambda_k e_k^2(t)
  \end{align*}
  and hence, by~\eqref{E:GKLT_4}, $E [ Y_t - Y_t^{(n)} ]^2 \longrightarrow 0$ as $n \rightarrow \infty$ uniformly in $t$ for $t \in C$. By the It\^o-Nisio Theorem (see Theorem~2.4 in~\cite{Vak87}), the convergence in quadratic mean implies the convergence almost surely.
\end{proof}


\subsection{A Karhunen-Lo\`eve expansion of the $\alpha$-Brownian bridge}

Let $(\lambda'_k)_{k=1}^\infty$ be the decreasing sequence of zeros of the function
\equ[E:Dist_3]{ F_{\alpha_0, \alpha_1, T}(\lambda) = \tan(\beta(\lambda) \ln(1-T)) + \lambda \beta(\lambda) / (1 + \lambda/2 - \lambda \alpha_0), }
where $\beta(\lambda) = \sqrt{(\alpha_0 + \alpha_1 - 1)/\lambda - \alpha_0(\alpha_0 - 1) - 1/4}$. Denote by $a_0 \wedge a_1$ and  $a_0 \vee a_1$ the minimum and maximum of $a_0$ and $a_1$. We prove the following theorem.

\begin{theorem}\label{T:KLE}
  Define
  \[ c := \frac{1/2 - (\alpha_0 \vee \alpha_1)}{ (1/2 - (\alpha_0 \vee \alpha_1)) \ln(1-T) + 1 } + 1/2. \]
  \begin{enumerate}
    \item If $\alpha_0 \wedge \alpha_1 \geq c$, then, under $\Prob^{(\alpha_0)}$, the sequence of decreasing eigenvalues in the Karhunen-Lo\`eve expansion of $X^{(\alpha)}$ under the measure $\mu_{\alpha_0, \alpha_1, T}$ is given by $(\lambda_k)_{k=1}^\infty$ = $(\lambda'_k)_{k=1}^\infty$. The corresponding normed eigenfunctions are given by
      \[ e_k(t) = \rho_k \sqrt{1-t} \sin(\beta(\lambda_k) \ln(1-t)), \]
      where $\rho_k$ is chosen such that $e_k$ is normalized in the $L_2(\mu_{\alpha_0, \alpha_1, T})$-norm.
    \item If $\alpha_0 \wedge \alpha_1 < c$, then, under $\Prob^{(\alpha_0)}$, the Karhunen-Lo\`eve expansion of $X^{(\alpha)}$ under the measure $\mu_{\alpha_0, \alpha_1, T}$ contains in addition to the eigenvalues from~(i) a further term $\lambda_0$ with
      \[ \lambda_0 = \frac{\alpha_0 + \alpha_1 - 1}{ \alpha_0(\alpha_0 - 1) + 1/4 - \sigma_0^2 } > \lambda_1, \]
      where $\sigma_0$ is the unique zero of the function
      \begin{align*}
        G(\sigma) &= (\sigma + \alpha_0 - 1/2) (\sigma + \alpha_1 - 1/2) \\
          &\qquad- (1-T)^{2\sigma} (\sigma - \alpha_0 + 1/2) (\sigma - \alpha_1 + 1/2)
      \end{align*}
      with $0 < \sigma_0 < 1/2 - \alpha_0$. The corresponding eigenfunction is given by
      \[ e_0(t) = \rho_0 \sqrt{1-t} \left( (1-t)^{\sigma_0} - (1-t)^{- \sigma_0} \right), \]
      where $\rho_0$ is chosen such that $e_0$ is normalized in the $L_2(\mu_{\alpha_0, \alpha_1, T})$-norm.
    \item If
      \[ \alpha_1 = 1 - \alpha_0 - \frac{(1-2\alpha_0)^2 \ln(1-T)}{2 (1-2\alpha_0) \ln(1-T) + 4} \]
      then, under $\Prob^{(\alpha_0)}$, the Karhunen-Lo\`eve expansion of $X^{(\alpha)}$ under the measure $\mu_{\alpha_0, \alpha_1, T}$ contains in addition to the eigenvalues from~(i) (and possibly~(ii)) a further term $\lambda_*$ with
      \[ \lambda_* = \frac{\alpha_0 + \alpha_1 - 1}{ \alpha_0(\alpha_0 - 1) + 1/4} > \lambda_1, \]
      The corresponding eigenfunction is given by
      \[ e_*(t) = \rho_* \sqrt{1-t} \ln(1-t), \]
      where $\rho_*$ is chosen such that $e_*$ is normalized in the $L_2(\mu_{\alpha_0, \alpha_1, T})$-norm. If~(ii) applies then $\lambda_* < \lambda_0$.
  \end{enumerate}
\end{theorem}

\begin{remark}
  Note that the constant $c$ is always less than or equal to $1/2$. Hence, if $\alpha_0, \alpha_1 \geq 1/2$ then the second part of Theorem~\ref{T:KLE} will never apply.
\end{remark}

The proof of Theorem~\ref{T:KLE} in full requires some simple but lengthy auxiliary calculations. They are organized as Lemmas~\ref{L:3} -- \ref{L:5} and moved to Section~\ref{S:Details} for the sake of readability. Moreover, we will carry out the proof only in the case $\alpha_0 \neq 1/2$. The case $\alpha_0 = 1/2$ leads to almost exactly the same calculations.

\begin{proof}[Proof for $\alpha_0 \neq 1/2$]
  Let $\lambda$ be a non-zero eigenvalue of the integral operator associated with the kernel $R^{(\alpha_0)}$ and the measure $\mu_{\alpha_0, \alpha_1, T}$, i.e., we consider the equation (with $t \in [0,T]$)
  \begin{align}
    \lambda e(t) &= \int_0^1 R^{(\alpha_0)}(t, s) e(s) \mu_{\alpha_0, \alpha_1, T}(ds) \label{E:New_1} \\
      &= (1-T)^{-1} R^{(\alpha_0)}(t, T) e(T) \notag \\
      &\qquad \qquad + (\alpha_0 + \alpha_1 - 1) \int_0^T R^{(\alpha_0)}(t, s) e(s) / (1-s)^2 ds, \notag
  \end{align}
  where $0 \neq e \in L_2(\mu_{\alpha_0, \alpha_1, T})$. Differentiating both sides of~\eqref{E:New_1} twice with respect to $t$ gives the second order differential equation
  \[ \lambda e''(t) = \lambda e(t) \alpha_0(\alpha_0-1)(1-t)^{-2} - (\alpha_0 + \alpha_1 - 1) e(t) (1-t)^{-2} \]
  or equivalently
  \equ[E:Dist_6]{ (1-t)^2 e''(t) - (\alpha_0(\alpha_0 - 1) - (\alpha_0 + \alpha_1 - 1) / \lambda) e(t) = 0 }
  with boundary conditions $e(0) = 0$ and
  \begin{align}
    \lambda e(T) &= (1-T)^{-1} R^{(\alpha_0)}(T, T) e(T) \label{E:Dist_7} \\
      &\qquad \qquad + (\alpha_0 + \alpha_1 - 1) \int_0^T R^{(\alpha_0)}(T, s) e(s) / (1-s)^2 ds. \notag
  \end{align}

  The general solution of~\eqref{E:Dist_6} is
  \[ e(t) = y(\ln(1-t)) = (1-t)^{1/2} \left( \rho (1-t)^\sigma + \varrho (1-t)^{-\sigma} \right), \]
  where $\sigma^2 = \sigma^2(\lambda)$ is given by
  \[ \sigma^2 = \alpha_0(\alpha_0 - 1) - (\alpha_0 + \alpha_1 - 1) / \lambda + 1/4. \]
  In fact, setting $e(t) = y(\ln(1-t))$ in~\eqref{E:Dist_6} together with the substitution $s = \ln(1-t)$ yields
  \begin{align}
    0 &= (1-t)^2 \frac{\partial^2}{\partial t^2} y(\ln(1-t)) - (\sigma^2 - 1/4) y(\ln(1-t)) \notag \\
      &= y''(\ln(1-t)) - y'(\ln(1-t)) - (\sigma^2 - 1/4) y(\ln(1-t)) \notag \\
      &= y''(s) - y'(s) - (\sigma^2 - 1/4) y(s). \label{E:New_2}
  \end{align}
  The characteristic polynomial
  \[ \chi(r) = r^2 - r - (\sigma^2 - 1/4) \]
  has roots $r_{1,2} = 1/2 \pm \sigma$.

  If $\sigma^2 = 0$ then $\lambda = (\alpha_0 + \alpha_1 - 1)/(1/4 + \alpha_0(\alpha_0 - 1))$ and there is a double root in $r = 1/2$ which yields the solution
  \[ y(s) = \rho s \exp(s/2) \]
  for~\eqref{E:New_2}. Then the solution of~\eqref{E:Dist_6} is
  \[ e(t) = \rho (1-t)^{1/2} \ln(1-t). \]
  The boundary condition $e(0) = 0$ is fulfilled and we have
  \[ \frac{\alpha_0 + \alpha_1 - 1}{1/4 + \alpha_0(\alpha_0 - 1)} e(T) = \int_0^1 R^{(\alpha_0)}(T, s) e(s) \mu_{\alpha_0, \alpha_1, T}(ds) \]
  if and only if
  \equ[E:New_3]{ \alpha_1 = 1 - \alpha_0 - \frac{(1-2\alpha_0)^2 \ln(1-T)}{2 (1-2\alpha_0) \ln(1-T) + 4}. }
  Thus, if~\eqref{E:New_3} is fulfilled then $\lambda_*$ with $\sigma^2(\lambda_*) = 0$ is an eigenvalue.

  Now assume $\sigma^2 \neq 0$. Then the general solution of~\eqref{E:New_2} is given by
  \[ y(s) = \rho \exp(r_1 s) + \varrho \exp(r_2 s). \]
  Hence, the solution of~\eqref{E:Dist_6} is
  \[ e(t) = y(\ln(1-t)) = \rho (1-t)^{1/2 + \sigma} + \varrho (1-t)^{1/2 - \sigma}. \]
  The boundary condition $e(0) = 0$ yields $\varrho = - \rho$ and thus
  \[ e(t) = \rho (1-t)^{1/2} \left( (1-t)^\sigma - (1-t)^{- \sigma} \right). \]
  By Lemma~\ref{L:3}, the boundary condition~\eqref{E:Dist_7} is fulfilled whenever
  \equ[E:Dist_9]{ 0 = \lambda \sigma ((1-T)^{\sigma} + (1-T)^{-\sigma}) + (1 + \lambda/2 - \lambda \alpha_0) ((1-T)^{\sigma} - (1-T)^{-\sigma}). }
  
  In the case $\sigma^2 < 0$ we have $\lambda < (\alpha_0 + \alpha_1 - 1)/(1/4 + \alpha_0(\alpha_0 - 1))$ and $\sigma = i \beta$ with
  \[ \beta = \sqrt{(\alpha_0 + \alpha_1 - 1)/\lambda - \alpha_0(\alpha_0 - 1) - 1/4} \in \R. \]
  It follows that
  \begin{align*}
    (1-t)^{\sigma} + (1-t)^{-\sigma} &= \exp(i \beta \ln(1-t) ) + \exp(- i \beta \ln(1-t) ) \\
      &= 2 \cos(\beta \ln(1-t)), \\
    (1-t)^{\sigma} - (1-t)^{-\sigma} &= 2 i \sin(\beta \ln(1-t)),
  \end{align*}
  and thus
  \[ e(t) = \rho' \sqrt{1-t} \sin(\beta \ln(1-t)). \]
  Then finding the solutions of~\eqref{E:Dist_9} is equivalent to finding the zeros in
  \[ F_{\alpha_0, \alpha_1, T}(\lambda) = \tan(\beta \ln(1-T)) + \lambda \beta / (1 + \lambda/2 - \lambda \alpha_0) \]
  with $\beta \neq 0$.

  In the case $\sigma^2 > 0$ we have $\lambda > (\alpha_0 + \alpha_1 - 1)/(1/4 + \alpha_0(\alpha_0 - 1))$. Expressing $\lambda$ in terms of $\sigma^2$,
  \[ \lambda = \frac{\alpha_0 + \alpha_1 - 1}{ \alpha_0(\alpha_0 - 1) + 1/4 - \sigma^2 }, \]
  we are looking for the non-zero solutions of
  \begin{align} 
    0 &= \frac{\sigma (\alpha_0 + \alpha_1 - 1)}{ \alpha_0(\alpha_0 - 1) + 1/4 - \sigma^2 } ((1-T)^{\sigma} + (1-T)^{-\sigma}) \notag \\
     &\qquad \qquad + (1 + \frac{(\alpha_0 + \alpha_1 - 1)(1/2 - \alpha_0)}{ \alpha_0(\alpha_0 - 1) + 1/4 - \sigma^2 } ) ((1-T)^{\sigma} - (1-T)^{-\sigma}) \notag \\
    &= \frac{ (1-T)^{-\sigma} }{(\alpha_0 - 1/2)^2 - \sigma^2 } G(\sigma), \label{E:Dist_10}
  \end{align}
  where
  \[ G(\sigma) = (\sigma + \alpha_0 - 1/2) (\sigma + \alpha_1 - 1/2) - (1-T)^{2\sigma} (\sigma - \alpha_0 + 1/2) (\sigma - \alpha_1 + 1/2). \]
  The first factor in~\eqref{E:Dist_10} is never equal to zero and thus we are looking for the non-zero zeros of $G$. Without loss of generality we may assume that $\alpha_0 < \alpha_1$ which implies $\alpha_1 > 1/2$ because of the assumption $\alpha_0 + \alpha_1 - 1 \geq 0$. Then if
  \[ \alpha_0 \geq \frac{1/2 - \alpha_1}{ (1/2 - \alpha_1) \ln(1-T) + 1 } + 1/2, \]
  it follows $G(\sigma) > 0$ for all $\sigma > 0$ by Lemma~\ref{L:4}. On the other hand, if
  \[ \alpha_0 < \frac{1/2 - \alpha_1}{ (1/2 - \alpha_1) \ln(1-T) + 1 } + 1/2 < 1/2, \]
  then by Lemma~\ref{L:5}, $G'(0) < 0$, $G(\sigma) > 0$ for $\sigma \geq 1/2 - \alpha_0$, and $G''(\sigma) > 0$ for $0 \leq \sigma \leq 1/2 - \alpha_0$, i.e., $G$ is strictly convex. This implies that $G$ has a unique zero with $0 < \sigma < 1/2 - \alpha_0$.
\end{proof}

We now show the summability of the eigenvalues $(\lambda_k)_{k=1}^\infty$.

\begin{prop}\label{P:SUMMABLE}
  There is a constant $c > 0$ such that $k^2 \lambda_k \longrightarrow c$ as $k$ tends to infinity. In particular $\sum_{k=1}^\infty \lambda_k < \infty$.
\end{prop}

\begin{proof}
  According to Theorem~\ref{T:KLE} the $\lambda_k$'s are (possibly except for two) given by the zeros of the function
  \equ[E:Dist_A4]{ F_{\alpha_0, \alpha_1, T}(\lambda) = \tan(\beta(\lambda) \ln(1-T)) + \lambda \beta(\lambda) / (1 + \lambda/2 - \lambda \alpha_0), }
  where $\beta(\lambda) = \sqrt{(\alpha_0 + \alpha_1 - 1)/\lambda - \alpha_0(\alpha_0 - 1) - 1/4}$. We know that $\lambda_k \longrightarrow 0$ and thus that
  \[ \lambda_k \beta(\lambda_k) / (1 + \lambda_k/2 - \lambda_k \alpha_0) \longrightarrow 0 \]
  as $k$ tends to infinity. Hence, for small $\lambda$ the zeros of $F_{\alpha_0, \alpha_1, T}(\lambda)$ are essentially given by the zeros of $\tan(\beta(\lambda) \ln(1-T))$. Those are given by $\lambda_k'$ such that $\beta(\lambda_k') \ln(1-T) = - k \pi$, $k=1,2,\ldots$, which implies that
  \[ \lambda_k' = \frac{\alpha_0 + \alpha_1 - 1}{\alpha_0(\alpha_0 - 1) + 1/4 + k^2 \pi^2 / (\ln(1-T))^2}. \]
  The result follows with $c = (\alpha_0 + \alpha_1 - 1)(\ln(1-T))^2 / \pi^2$.
\end{proof}


\section{Ornstein-Uhlenbeck processes}\label{S:OU}

The approach described in Section~\ref{S:Intro} is not restricted to $\alpha$-Brownian bridges but may be applied to other cases as well. We briefly study the case of Ornstein-Uhlenbeck processes. Proofs are omitted since they follow the same paths as the proofs of Proposition~\ref{P:LRP} and Theorem~\ref{T:KLE}. Moreover, in order to emphasize the analogy to Section~\ref{S:Intro} we use the same notation as there.

We consider the stochastic differential equation
\[ dX^{(\alpha)}_t = dW_t - \alpha X^{(\alpha)}_t dt, \qquad X^{(\alpha)}_0 = 0, \quad 0 \leq t < \infty, \]
where $\alpha \geq 0$ and $W = (W_t)_{t \in [0,\infty)}$ is standard Brownian motion. Let $\F_t$ be the induced filtration of $(W_s)_{s \in [0,t]}$. Given two different values $\alpha_0, \alpha_1 \geq 0$ and some time $0 < T < \infty$, we want to test
\[ H_0: \alpha = \alpha_0 \qquad \text{vs.} \qquad H_1: \alpha = \alpha_1, \]
based on an observed trajectory of $X^{(\alpha)}$ until time $T$. Again, our aim is to find that decision which minimizes the probability of making an error of the second kind, given that the probability of making an error of the first kind is not larger than $q$ for some $0 \leq q \leq 1$. Let $\Prob^{(\alpha)}_T$ be the induced measure of $(X^{(\alpha)}_t)_{t \in [0,T]}$ on the measurable space $(C([0,T]), \mathcal{C})$. Then we have to decide according to the following rule:
\[ \text{reject $H_0$ if $\phi_{\alpha_0, \alpha_1}(T) > c_{\alpha_0, \alpha_1, T}(q)$,} \]
where $\phi_{\alpha_0, \alpha_1}(T) := d\Prob^{(\alpha_1)}_T / d\Prob^{(\alpha_0)}_T$ is the likelihood ratio at time $T$ and $c_{\alpha_0, \alpha_1, T}(q)$ is chosen such that
\[ \Prob^{(\alpha_0)}(\phi_{\alpha_0, \alpha_1}(T) > c_{\alpha_0, \alpha_1, T}(q)) = q. \]

Analogous to Proposition~\ref{P:LRP} we obtain
\equ[E:Phi_OU]{ \phi_{\alpha_0, \alpha_1}(T) = \exp \left( (\alpha_0 - \alpha_1) ( \psi_{\alpha_0, \alpha_1}(T) - T )/2 \right), }
where
\[ \psi_{\alpha_0, \alpha_1}(T) = (X^{(\alpha)}_T)^2 + (\alpha_0 + \alpha_1) \int_0^T (X^{(\alpha)}_s)^2 ds. \]
Introducing the measure
\[ \mu_{\alpha_0, \alpha_1, T}(ds) := \delta_T(ds) + (\alpha_0+\alpha_1) \Ind(s \leq T) ds, \]
we obtain
\[ \psi_{\alpha_0, \alpha_1}(T) = \| X^{(\alpha)} \|^2_{L_2(\mu_{\alpha_0, \alpha_1, T})}. \]

The covariance function of $X^{(\alpha)}$ is given by
\[ R^{(\alpha)}(s,t) := \E[ X^{(\alpha)}_s X^{(\alpha)}_t ] = \frac{1}{2\alpha}(e^{-\alpha |s-t|} - e^{-\alpha (s+t)}). \]
With $R^{(\alpha_0)}$ we associate the integral operator $A_{R^{(\alpha_0)}}: L_2(\mu_{\alpha_0, \alpha_1, T}) \rightarrow L_2(\mu_{\alpha_0, \alpha_1, T})$ defined by
\[ (A_{R^{(\alpha_0)}}e)(t) = \int_0^\infty R^{(\alpha_0)}(t,s) e(s) \mu_{\alpha_0, \alpha_1, T}(ds). \]
Analogous to Theorem~\ref{T:KLE} we obtain

\begin{theorem} 
  The sequence $(\lambda_k)_{k=1}^\infty$ of decreasing eigenvalues of the operator $A_{R^{(\alpha_0)}}$ is given by the zeros of the function
  \[ F_{\alpha_0, \alpha_1, T}(\lambda) := \tan(\beta(\lambda)T) - \frac{\lambda \beta(\lambda)}{1 - \lambda \alpha_0}, \]
  where $\beta(\lambda) = \sqrt{(\alpha_0+\alpha_1)/\lambda - \alpha_0^2}$. The corresponding normed eigenfunctions are given by
  \[ e_k(t) = \rho_k \sin(\beta(\lambda_k) t), \]
  where $\rho_k$ is chosen such that $e_k$ is normalized in the $L_2(\mu_{\alpha_0, \alpha_1, T})$-norm.
\end{theorem}

By Theorem~\ref{T:GKLE} we have the following series expansion of $X^{(\alpha_0)}$:
\equ[E:KLE_OU]{ X^{(\alpha_0)}_t = \sum_{k=0}^\infty Z_k e_k(t), }
where $(Z_k)_{k=0}^\infty$ is a sequence of independent normal random variables with $\E Z_k^2 = \lambda_k$. The convergence in~\eqref{E:KLE_OU} is almost surely uniform in $t$ for all $t \in [0,T]$. It follows that
\equ[E:Psi_Dist_Equ_OU]{ \psi_{\alpha_0, \alpha_1}(T) = \| X^{(\alpha_0)} \|^2_{L_2(\mu_{\alpha_0, \alpha_1, T})} = \sum_{k=0}^\infty Z_k^2 =_d \sum_{k=0}^\infty \lambda_k \mathcal{N}_k^2, }
where $(\mathcal{N}_k)_{k=0}^\infty$ is an i.i.d. sequence of standard normal random variables.

Hence, according to~\eqref{E:Psi_Dist_Equ_OU}, the distribution function of $\psi_{\alpha_0, \alpha_1}(T)$ under $\Prob^{(\alpha_0)}$ is given by~\eqref{E:Dist_Q} with $\nu_k = \lambda_{k-1}$. Finally, from~\eqref{E:Phi_OU} we obtain

\begin{theorem}
  If $\alpha_0 < \alpha_1$, then $\phi_{\alpha_0, \alpha_1}(T) \leq \exp((\alpha_1 - \alpha_0)T/2)$ and the distribution function of $\phi_{\alpha_0, \alpha_1}(T)$ under $\Prob^{(\alpha_0)}$ is given by $\Prob^{(\alpha_0)}(\phi_{\alpha_0, \alpha_1}(T) \leq x) = D_{\alpha_0, \alpha_1, T}(x)$, where
  \[ D_{\alpha_0, \alpha_1, T}(x) = \frac{1}{\pi} \sum_{k=0}^\infty (-1)^k \int_{1/\lambda_{2k}}^{1/\lambda_{2k+1}} \frac{e^{-uT/2} x^{u/(\alpha_1 - \alpha_0)}}{u \sqrt{|F(u)|}} du \]
  with $F(u) = \prod_{l=0}^\infty ( 1 - \lambda_l u )$. In the case $\alpha_0 > \alpha_1$ we have $\phi_{\alpha_0, \alpha_1}(T) \geq \exp((\alpha_1 - \alpha_0)T/2)$ and $\Prob^{(\alpha_0)}(\phi_{\alpha_0, \alpha_1}(T) \leq x) = 1 - D_{\alpha_0, \alpha_1, T}(x)$.
\end{theorem}


\section{Remaining proofs}\label{S:Details}

\begin{lemma}\label{L:3}
  Assume that $0 \leq \alpha_0 \neq 1/2$ and let
  \equ[E:Dist_8n]{ e(t) = \rho (1-t)^{1/2} \left( (1-t)^\sigma - (1-t)^{- \sigma} \right), }
  where $\sigma^2 = \sigma^2(\lambda) = \alpha_0(\alpha_0 - 1) - (\alpha_0 + \alpha_1 - 1)/\lambda + 1/4$. Then, for $0 < T < 1$,
  \equ[E:Dist_7n]{ \lambda e(T) = \int_0^1 R^{(\alpha_0)}(T, s) e(s) \mu_{\alpha_0, \alpha_1, T}(ds) }
  if and only if
  \[ 0 = \lambda \sigma ((1-T)^{\sigma} + (1-T)^{-\sigma}) + (1 + \lambda/2 - \lambda \alpha_0) ((1-T)^{\sigma} - (1-T)^{-\sigma}). \]
\end{lemma}

\begin{proof}
  Since we assume $\alpha_0 \neq 1/2$ we can multiply both sides in~\eqref{E:Dist_7n} by $1-2\alpha_0$. Moreover, we ignore the constant $\rho \neq 0$ in~\eqref{E:Dist_8n}. That is, with the definition of $\mu_{\alpha_0, \alpha_1, T}$ in~\eqref{E:Mu}, we consider
  \begin{align}
    0 &= - \lambda (1-2\alpha_0) \tilde e(T) + \frac{ (1-2\alpha_0) R^{(\alpha_0)}(T, T) \tilde e(T) }{1-T} \label{E:Dist_6n} \\
    &\qquad + (\alpha_0 + \alpha_1 - 1) (1-2\alpha_0) \int_0^T R^{(\alpha_0)}(T, s) \tilde e(s) / (1-s)^2 ds \notag
  \end{align}
  with
  \equ[E:Dist_5n]{ \tilde e(t) = (1-t)^{\sigma + 1/2} - (1-t)^{-\sigma + 1/2}. }

  Then
  \begin{align*}
    &(1-2\alpha_0) \int_0^T R^{(\alpha_0)}(T, s) \tilde e(s) / (1-s)^2 ds \\
    &\quad = \frac{ (1-T)^{1/2 - \sigma} - (1-T)^{2\alpha_0 - 1/2 + \sigma} }{\alpha_0 - 1/2 + \sigma} + \frac{(1-T)^{2\alpha_0 - 1/2 - \sigma} - (1-T)^{1/2 + \sigma}}{\alpha_0 - 1/2 - \sigma} \\
    &\quad = \frac{ (\alpha_0 - 1/2 - \sigma) \left( (1-T)^{1/2 - \sigma} - (1-T)^{2\alpha_0 - 1/2 + \sigma} \right)}{(\alpha_0 + \alpha_1 - 1)/\lambda} \\
    &\quad \qquad + \frac{ (\alpha_0 - 1/2 + \sigma) \left( (1-T)^{2\alpha_0 - 1/2 - \sigma} - (1-T)^{1/2 + \sigma} \right)}{(\alpha_0 + \alpha_1 - 1)/\lambda},
  \end{align*}
  where we used $(\alpha_0 - 1/2 + \sigma)(\alpha_0 - 1/2 - \sigma) = (\alpha_0 + \alpha_1 - 1)/\lambda$. Plugging this into~\eqref{E:Dist_6n} and replacing $\tilde e(t)$ according to~\eqref{E:Dist_5n} we get
  \begin{align*}
    0 &= - \lambda (1-2 \alpha_0) \left( (1-T)^{1/2 + \sigma} - (1-T)^{1/2 - \sigma} \right) \\
      &\qquad + \lambda (\alpha_0 - 1/2 - \sigma) \left( (1-T)^{1/2 - \sigma} - (1-T)^{2\alpha_0 - 1/2 + \sigma} \right) \\
      &\qquad + \lambda (\alpha_0 - 1/2 + \sigma) \left( (1-T)^{2\alpha_0 - 1/2 - \sigma} - (1-T)^{1/2 + \sigma} \right) \\
      &\qquad + \left( (1-T)^{1/2+\sigma} - (1-T)^{1/2 - \sigma} \right) \left( (1-T)^{2\alpha_0 - 1} - 1 \right) \\
      &= \lambda \sigma \left( (1-T)^\sigma + (1-T)^{- \sigma} \right) \left( (1-T)^{2 \alpha_0 - 1/2} - (1-T)^{1/2} \right) \\
      &\qquad + (1+\lambda(1/2 - \alpha_0)) \left( (1-T)^\sigma - (1-T)^{- \sigma} \right) \\
      &\qquad \qquad \times \left( (1-T)^{2 \alpha_0 - 1/2} - (1-T)^{1/2} \right).
  \end{align*}
  We divide by $(1-T)^{2 \alpha_0 - 1/2} - (1-T)^{1/2}$ and obtain
  \[ 0 = \lambda \sigma ((1-T)^{\sigma} + (1-T)^{-\sigma}) + (1 + \lambda/2 - \lambda \alpha_0) ((1-T)^{\sigma} - (1-T)^{-\sigma}). \qedhere \]
\end{proof}

In Lemmas~\ref{L:4} and~\ref{L:5} we consider the function
\[ G(\sigma) = (\sigma + \alpha_0 - 1/2) (\sigma + \alpha_1 - 1/2) - (1-T)^{2\sigma} (\sigma - \alpha_0 + 1/2) (\sigma - \alpha_1 + 1/2) \]
for $\sigma \geq 0$.

\begin{lemma}\label{L:4}
  Assume that $0 \leq T < 1$, $\alpha_1 > 1/2$, and
  \[ \alpha_0 \geq \frac{1/2 - \alpha_1}{ (1/2 - \alpha_1) \ln(1-T) + 1 } + 1/2. \]
  Then $G(\sigma) > 0$ for all $\sigma > 0$.
\end{lemma}

\begin{proof}
  We have
  \begin{align*}
    G(\sigma) &= \alpha_0 (\sigma(1+(1-T)^{2\sigma}) + (1-(1-T)^{2\sigma})(\alpha_1 - 1/2)) \\
              &\qquad + (1-(1-T)^{2\sigma})(\sigma^2 - \alpha_1/2 + 1/4) + \sigma(1+(1-T)^{2\sigma})(\alpha_1-1) \\ 
              &\geq \left( \frac{1/2 - \alpha_1}{ (1/2 - \alpha_1) \ln(1-T) + 1 } + 1/2 \right) \\
              &\qquad \qquad \times (\sigma(1+(1-T)^{2\sigma}) + (1-(1-T)^{2\sigma})(\alpha_1 - 1/2)) \\
              &\qquad + (1-(1-T)^{2\sigma})(- \alpha_1/2 + 1/4) + \sigma(1+(1-T)^{2\sigma})(\alpha_1-1) \\
              &= - \frac{(\alpha_1-1/2)^2}{(1/2 - \alpha_1) \ln(1-T) + 1} f(\sigma)
  \end{align*}
  with
  \[ f(\sigma) = \sigma(1+(1-T)^{2\sigma}) \ln(1-T) + (1-(1-T)^{2\sigma}). \]
  It follows $f(0) = 0$ and
  \[ f'(\sigma) = \ln(1-T) \left( 1- (1-T)^{2\sigma} + 2 \sigma \ln(1-T) (1-T)^{2\sigma} \right) \]
  implying $f'(0) < 0$. Moreover,
  \[ f''(\sigma) = 4 \sigma \ln(1-T)^3 (1-T)^{2\sigma} < 0, \]
  i.e., $f$ is a strictly concave function and thus, $f(\sigma) < 0$ for all $\sigma > 0$. This yields,
  \[ G(\sigma) \geq - \frac{(\alpha_1-1/2)^2}{(1/2 - \alpha_1) \ln(1-T) + 1} f(\sigma) > 0 \]
  for all $\sigma > 0$.
\end{proof}

\begin{lemma}\label{L:5}
  Assume $0 \leq T < 1$, $\alpha_1 > 1/2$, and
  \[ \alpha_0 < \frac{1/2 - \alpha_1}{ (1/2 - \alpha_1) \ln(1-T) + 1 } + 1/2. \]
  Then
  \begin{enumerate}
    \item $G(\sigma) > 0$ for $\sigma \geq 1/2 - \alpha_0$,
    \item $G'(0) < 0$, and
    \item $G''(\sigma) > 0$ for $0 \leq \sigma \leq 1/2 - \alpha_0$.
  \end{enumerate}
\end{lemma}

\begin{proof}
  If $(\sigma - \alpha_0 + 1/2) (\sigma - \alpha_1 + 1/2) > 0$, the estimate $(1-T)^{2\sigma} < 1$ yields
  \begin{align*}
    G(\sigma) &> (\sigma + \alpha_0 - 1/2)(\sigma + \alpha_1 - 1/2) - (\sigma - \alpha_0 + 1/2)(\sigma - \alpha_1 + 1/2) \\
      &= 2 \sigma (\alpha_0 + \alpha_1 - 1) \geq 0.
  \end{align*}
  If $(\sigma - \alpha_0 + 1/2) (\sigma - \alpha_1 + 1/2) < 0$ the estimate $(1-T)^{2\sigma} > 0$ yields
  \begin{align*}
    G(\sigma) > (\sigma + \alpha_0 - 1/2)(\sigma + \alpha_1 - 1/2) \geq 0,
  \end{align*}
  since $(\sigma + \alpha_0 - 1/2) \geq 0$ and $(\sigma + \alpha_1 - 1/2) \geq 0$ by the assumptions. This proves~(i).

  The derivative of $G$ is given by
  \begin{align*}
    G'(\sigma) &= 2 \sigma (1-(1-T)^{2\sigma}) + (\alpha_0 + \alpha_1 - 1)(1+(1-T)^{2\sigma}) \\
      &\qquad-2(1-T)^{2\sigma}\ln(1-T) (\sigma - \alpha_0 + 1/2) (\sigma - \alpha_1 + 1/2).
  \end{align*}
  Hence,
  \begin{align*}
    G'(0) &= 2 (\alpha_0 + \alpha_1 - 1) - 2 \ln(1-T) (1/2 - \alpha_0) (1/2 - \alpha_1) \\
          &= 2 \alpha_0 (1+\ln(1-T)(1/2 - \alpha_1)) + 2 \alpha_1 - 2  - \ln(1-T) (1/2 - \alpha_1) \\
          &> 2 \left(\frac{1/2 - \alpha_1}{ (1/2 - \alpha_1) \ln(1-T) + 1 } + 1/2\right)(1+\ln(1-T)(1/2 - \alpha_1)) \\
          &\qquad \qquad + 2 \alpha_1 - 2  - \ln(1-T) (1/2 - \alpha_1) \\
          &= 2 (1/2 - \alpha_1) + (1/2 - \alpha_1) \ln(1-T) \\
          &\qquad \qquad + 1 + 2 \alpha_1 - 2  - \ln(1-T) (1/2 - \alpha_1) \\
          &= 0,
  \end{align*}
  which proves~(ii).

  The second derivative of $G$ equals
  \[ G''(\sigma) = 2(1-(1-T)^{2\sigma}) + 4 (1-T)^{2\sigma} |\ln(1-T)| g(\sigma) \]
  with
  \[ g(\sigma) = 2\sigma - \alpha_0 - \alpha_1 + 1 + \ln(1-T)(\sigma-\alpha_0 + 1/2)(\sigma-\alpha_1 + 1/2). \]

  Assume $\sigma \geq (\alpha_0 + \alpha_1 - 1)/2$. The assumption $\sigma \leq 1/2 - \alpha_0$ implies $\sigma \geq \alpha_1 - 1/2$ and thus $\ln(1-T)(\sigma-\alpha_0 + 1/2)(\sigma-\alpha_1 + 1/2) \geq 0$. It follows
  \[ g(\sigma) \geq 2\sigma - \alpha_0 - \alpha_1 + 1 \geq 0. \]

  Finally, if $\sigma < (\alpha_0 + \alpha_1 - 1)/2$, then
  \begin{align*} 
    g(0) &= 1 - \alpha_0 - \alpha_1 + \ln(1-T)(1/2-\alpha_0)(1/2 - \alpha_1) \\
      &> 1 - \alpha_0 - \alpha_1 - \frac{\ln(1-T)(1/2 - \alpha_1)^2}{ (1/2 - \alpha_1) \ln(1-T) + 1 } \\
      &> 1 - \alpha_0 - \alpha_1 - \frac{\ln(1-T)(1/2 - \alpha_1)^2}{ (1/2 - \alpha_1) \ln(1-T)} \\
      &= 1 - \alpha_0 - \alpha_1 - (1/2 - \alpha_1) = 1/2 - \alpha_0 > 0
  \end{align*}
  and
  \[ g'(\sigma) = 2 + \ln(1-T)(2\sigma - \alpha_0 - \alpha_1 + 1) \geq 2. \]
  Hence $g(\sigma) \geq 0$ for all relevant $0 \leq \sigma \leq 1/2 - \alpha_0$ and thus $G''(\sigma) > 0$ for all $0 \leq \sigma \leq 1/2 - \alpha_0$. This proves~(iii).
\end{proof}


\section*{Acknowledgments}

The author would like to thank Ingemar Kaj for valuable comments.
 

\bibliographystyle{plain}
\bibliography{bibl}


\end{document}